\documentclass[12pt,a4paper, reqno]{article}

\usepackage{indentfirst} % identation on the first paragraph
\usepackage{amsmath, amsthm, amssymb,xfrac, mathrsfs} % math packages
\usepackage{tikz} % figures
\usetikzlibrary{hobby}
\usetikzlibrary{decorations.pathreplacing}
\usepackage{wrapfig} % allows to put text aroung figures
\usepackage{xcolor} % colors
\usepackage{verbatim} % multiline comments
\usepackage{enumerate} % making lists

\usepackage{hyperref} % insert this allways last
\hypersetup{
colorlinks=true,
citecolor=black!50!red,
linkcolor=black!50!red,
linktoc=all
}
\usepackage[all]{hypcap} % to correct positioning for the hyperref

\newcounter{constant}
\newcommand{\newconstant}[1]{\refstepcounter{constant}\label{#1}}
\newcommand{\useconstant}[1]{c_{\textnormal{\tiny \ref{#1}}}}
\newcounter{bigconstant}

\newtheorem{teo}{Theorem}[section]
\newtheorem{prop}[teo]{Proposition}
\newtheorem{lemma}[teo]{Lemma}

\newtheorem{cor}[teo]{Corollary}
\newtheorem{conjecture}[teo]{Conjecture}

\numberwithin{equation}{section} % numbering according to sections

\theoremstyle{definition}

\newtheorem{remark}[teo]{Remark}

\newcommand{\PP}{\mathbb{P}}
\newcommand{\EE}{\mathbb{E}}
\newcommand{\RR}{\mathbb{R}}

\newcommand{\NN}{\mathbb{N}}
\newcommand{\ZZ}{\mathbb{Z}}

\newcommand{\charf}[1]{\mathbf{1}_{#1}}

\DeclareMathOperator{\dist}{d}

\DeclareMathOperator{\poisson}{Poisson}
\DeclareMathOperator{\ber}{Bernoulli}
\DeclareMathOperator{\expo}{Exponential}
\DeclareMathOperator{\unif}{U}
\DeclareMathOperator{\cov}{Cov}

\DeclareMathOperator{\cir}{Circ}

\newcommand{\drawing}{\vcenter{\hbox{\;\begin{tikzpicture}
\draw[thick] (-2,0) rectangle (2,2);
\draw[thick] (0,0) -- (0,2);
\draw[thick] (1,0) -- (1,2);
\draw[thick] (-1,0) -- (-1,2);
\draw[thick] (-0.15,0.85) rectangle (0.15,1.15);

\draw [decorate,decoration={brace,amplitude=5,raise=0.4},yshift=0pt, thick, rotate=90] (-0.05, -2) -- (-0.05, 2);
\node[below] at (0,-0.2){$2n$};
\draw [decorate,decoration={brace,amplitude=5,raise=0.4},yshift=0pt, thick, rotate=270] (-2.05, -1) -- (-2.05, 1);
\node[above] at (0,2.2){$n$};
\draw [decorate,decoration={brace,amplitude=5,raise=0.4},yshift=0pt, thick] (-2.05, 0) -- (-2.05, 2);
\node[left] at (-2.2,1){$n$};

\draw[thick] (0, 1) .. controls (0.9, 2) and (0.5, -0.3) .. (2, 0.1);
\draw[thick] (0, 0.9) .. controls (-2.5, 0.3) and (0.5, 2) .. (-2, 1.5);

\draw[thick, shift={(-1,1)}] plot[smooth,tension=.2]
  coordinates{(0.4,0) (0.3,0.6)(1, 0.5)(1.5,0.8)(1.7,0.5)(1.4,0)(1.7,-0.5)(1.5,-0.7)(1.2,-0.5)(1,-0.5)(0.7,-0.8)(0.4,-0.4)(0.6,-0.2)(0.4,0)};

\end{tikzpicture}
}}
}

\begin{document}
% Text

\title{Percolation in majority dynamics}

\author{Gideon Amir\footnote{Email: \ gidi.amir@gmail.com; \ Bar-Ilan University, 5290002, Ramat Gan, Israel} \and Rangel Baldasso\footnote{Email: \ baldasso@impa.br; \ Bar-Ilan University, 5290002, Ramat Gan, Israel}}

\maketitle

\begin{abstract}
We consider two-dimensional dependent dynamical site percolation where sites perform majority dynamics. We introduce the critical percolation function at time $t$ as the infimum density with which one needs to begin in order to obtain an infinite open component at time $t$. We prove that, for any fixed time $t$, there is no percolation at criticality and that the critical percolation function is continuous. We also prove that, for any positive time, the percolation threshold is strictly smaller than the critical probability for independent site percolation.
\end{abstract}

\begin{comment}
\subjclass[2010]{82C43, 82B43, 82C22}
\keywords{Majority dynamics, percolation.}
\end{comment}

\section{Introduction}\label{sec:intro}
~
\par Since its introduction in Broadbent and Hammersley~\cite{bh}, percolation theory has become one of the most prominent areas of probability: It was considered a model for the spread of a fluid through a random medium but nowadays its applications permeate many different fields. Not only that, but many interesting theoretical discoveries where made regarding these models. In this paper we work on a modification of dynamical site percolation in $\ZZ^{2}$, where open and closed vertices perform majority dynamics.

\par In two-dimensional majority dynamics, each site $x \in \ZZ^{2}$ initially receives an opinion that can be either 0 or 1. After an exponential random time, the vertex pools its neighbors' opinions together with its own and chooses the most common one. Denote by $\eta_{t}$ the configuration at time $t \geq 0$.

\par We choose the initial opinions in an i.i.d. manner with marginal distribution $\ber(p)$, where $p \in [0,1]$ is fixed, and denote by $\PP_{p}$ the distribution of the process in this case.

\par Contrary to the usual dynamical percolation model, where sites perform independent spin-flip dynamics, in our case the evolution is not stationary nor independent. In particular, at any fixed time $t>0$, the configuration presents non-trivial dependencies. Our goal is to understand the percolative properties of the system run up to time $t$. We say that a configuration $\eta_{t}$ percolates if it contains an infinite connected component of constant opinion 1 and introduce the critical probability for time $t$ as
\begin{equation}\label{eq:critical_percolation}
p_{c}(t)=\inf\{p \in [0,1]: \PP_{p}[\eta_{t} \text{ percolates}] >0\}.
\end{equation}

\par Our main result states that $\eta_{t}$ does not percolate at criticality.
\begin{teo}\label{t:critical_probability}
For any $t \geq 0$, 
\begin{equation}
\PP_{p_{c}(t)}[\eta_{t} \text{ percolates}]=0.
\end{equation}
\end{teo}

\par The choice of $\ZZ^{2}$ as our underlying graph plays an important role. In this case, cycles of constant opinion are stable structures for the dynamics. Not only that, but infinite paths of constant opinion can only be destroyed from one side and this happens with finite speed. This implies that if percolation occurs for some $t$ and $p$, then it happens for all $s \geq t$. In particular, we have
\begin{equation}
p_{c}(t) \leq p_{c}(s), \text{ for all } t \geq s.
\end{equation}

\par Heuristically, the model presents some clustering phenomena that should help the construction of big connected components. This is an evidence that the critical probabilities should strictly decrease in time. However, this is not completely clear, since the clustering works in both directions: Components of opinion $0$ can also increase in size and prevent the creation of infinite clusters of 1s. We provide a partial result.
\begin{teo}\label{t:strictly_decreasing_pc}
For any $t >0$,
\begin{equation}
\frac{1}{2} \leq p_{c}(t) < p_{c}(0) = p_{c}^{site},
\end{equation}
where $p_{c}^{site}$ denotes the critical probability for two-dimensional independent site percolation.
\end{teo}

\par This says that majority dynamics affects the percolative phase in a non-trivial manner. Our last result says that this perturbation is continuous in time.
\begin{teo}\label{t:continuity}
The function $t \mapsto p_{c}(t)$ is continuous.
\end{teo}

\textbf{Proof overview.} To conclude Theorem~\ref{t:critical_probability}, we analyze the supercritical regime. We prove that the set
\begin{equation}\label{eq:set_P}
\mathscr{P}=\{(p,t) \in [0,1] \times \RR_{+}: \PP_{p}[\eta_{t} \text{ percolates}]>0\}
\end{equation}
is open in $[0,1] \times \RR_{+}$.

\par The understanding of critical phenomena in percolation theory is closely related to the study of rectangle crossings, and here this relation is also used. For $n \in \NN$ and $\lambda > 0$, define the crossing event $H(\lambda n,n)$ as the existence of an open crossing of the rectangle $[1,\lfloor \lambda n \rfloor] \times [1,n]$ connecting the left boundary $\{1\} \times [1,n]$ to the right boundary $\{\lfloor \lambda n \rfloor \} \times [1,n]$.

\par We first consider the crossing events for parameters $(p,t)$ in $\mathscr{P}$. In this case, we prove that
\begin{equation}
\PP_{p}[\eta_{t} \in H(\lambda n, n)] \to 1, \text{ as } n \to \infty,
\end{equation}
for any $\lambda>0$. This result is proved with the help of a Russo-Seymour-Welsh theory borrowed from Tassion~\cite{tassion}.

\par The second step of the proof connects the crossing events with the existence of percolation. We will verify that if $\PP_{p}[\eta_{t} \in H(3n, n)]$ is large enough for some big value of $n$ depending on $p$ and $t$, then $(p,t) \in \mathscr{P}$. We use multiscale renormalisation to prove that, provided $\PP_{p}[\eta_{t} \in H(3n, n)]$ is sufficiently big, it converges exponentially fast to one. This will allow the construction of an infinite open path with positive probability.

\par Our approach towards Theorem~\ref{t:critical_probability} is very general and should apply to many different two-dimensional models. We believe it is possible to prove the same result for any model that satisfies
\begin{itemize}
\item rotation, reflection and translation invariance;
\item the FKG inequality;
\item large-degree polynomial decay of spatial correlations (see~\eqref{eq:decoupling} for correlation decay in the case of majority dynamics);
\item continuity of the crossing probabilities as a function of $p$.
\end{itemize}

\par The proof of Theorem~\ref{t:continuity} also uses that the set $\mathscr{P}$ is open, since we may write
\begin{equation}
p_{c}(t)=\inf\{ p \in [0,1]: (p,t) \in \mathscr{P}\}.
\end{equation}
Together with the fact that $t \mapsto p_{c}(t)$ is non-increasing, it is easy to conclude that this function is left-continuous. To conclude the continuity from the right, we construct a coupling between two majority dynamics with different initial densities that allows us to compare the percolative phases in two different times.

\par As for Theorem~\ref{t:strictly_decreasing_pc}, our approach is based on the enhancement theorem, see Aizenman and Grimmett~\cite{ag} and Balister, Bollobás and Riordan~\cite{bbr}. The enhancement used here is based on Camia, Newman and Sidoravicius~\cite{cns}. The idea is to prove that it is possible to create an infinite component with subpaths that are initially open, whose endpoints do not ring up to time $t$ and are at distance two from each other. In particular, there is a vertex that connects two consecutive endpoints and, when the enhancement is performed in this vertex, the connected component increases in size.

\textbf{Related works.} The continuous time model, sometimes called the asynchronous model, is usually considered together with its discrete counterpart, where all sites are uptated at the same time. Many works consider both models. In Moran~\cite{moran} and Ginosar and Holzman~\cite{gh}, the authors study the dynamics on bounded degree graphs. They prove that, provided the graph does not grow very fast, it presents the period-two property, that says each vertex eventually has an orbit of period at most two. In Tamuz and Tessler~\cite{tt}, this result is strengthened, proving that, in the asynchronous model, almost surely each vertex eventually fixates and, besides, that it changes opinion only a bounded number of times. This last result is actually combinatorial: If no two clocks ring at the same time, then, for any initial condition, each site eventually fixates.

\par It is not the case that, for all bounded degree graphs, any initial configuration presents the period-two property under discrete time majority dynamics. One such example where this does not happen is the $d$-regular infinite tree. It is not hard to construct an initial configuration where this does not hold. Even so, Benjamini, Chan, O'Donnell, Tamuz and Tan~\cite{bcott} proves that, for unimodular transitive graphs, if the initial distribution of opinions is invariant with respect to the automorphism group of the graph, then the period-two property occurs almost surely.

\par Back to the asynchronous model, the speed with which fixation occurs may depend on the initial density and it is still not fully understood. It is believed that, with i.i.d. initial condition with density $p$, the probability of not fixating until time $t$ decays exponentially fast. A partial result in this direction is given in~\cite{cns}, where the authors obtain stretched exponential bounds for this probability when the underlying graph is the hexagonal lattice, provided the initial density is not close to $\frac{1}{2}$.

\par Perhaps the most natural modification of the model is changing the way draws are settled by choosing a new independent opinion uniformly at random. This is known as zero-temperature Glauber dynamics (ZTGD) for the Ising model and here the behavior is dramatically different. To exemplify, let $(\xi_{t})_{t \geq 0}$ denote the ZTGD in $\ZZ^{d}$ and define
\begin{equation}\label{eq:fixation_ztgd}
\tilde{p}_{c}(d)= \inf\{p \geq 0: \PP_{p}[\lim_{t \to \infty}\xi_{t}(0)=1]=1\},
\end{equation} 
the threshold for fixation at opinion 1. The expected behavior of $\tilde{p}_{c}(d)$ is stated in the following conjecture.
\begin{conjecture}
For any $d \geq 2$, $\tilde{p}_{c}(d)=\frac{1}{2}$.
\end{conjecture}
As a consequence of Arratia~\cite{arratia}, one easily obtains $\tilde{p}_{c}(1)=1$. However, very little is known when $d \geq 2$. Fontes, Schonmann and Sidoravicius~\cite{fss} proved that $\tilde{p}_{c}(d) \in (0,1)$, for all $d \geq 2$, while Morris~\cite{morris} established that $\tilde{p}_{c}(d) \to \frac{1}{2}$ as $d$ increases.

\par When consideing fixation in this model, ergodic arguments yield that fixation implies unanimity of opinions, since the only stable structures are the whole space, halfspaces and strips. Here, the case $p=\frac{1}{2}$ in $\ZZ^{2}$ presents a distinctive behavior, as proved by Nanda, Newman and Stein~\cite{nns}: Almost surely, no vertex fixates, contrary to majority dynamics. Camia, De Santis and Newman~\cite{csn} studied finer properties of the model and proved that the flip rate at any given vertex converges to zero in probability. We remark that these results are not known for $\ZZ^{d}$, with $d \geq 3$.

\par When the model evolves on top of $d$-regular trees, one can introduce the critical threshold $\tilde{p}_{c}^{tree}(d)$ in the same way as in~\eqref{eq:fixation_ztgd}. Here, Howard~\cite{howard} established that $\tilde{p}_{c}(3) > \frac{1}{2}$, while Caputo and Martinelli~\cite{cm} proved that $\tilde{p}_{c}(d) \to \frac{1}{2}$ as $d$ grows.

%%%%%
\newconstant{c:fixation_decay}
%%%%%

\textbf{Conjectures and open problems.} The evolution in majority dynamics is very localized and has many stable structures that should appear very fast. However, the decay of the probability of the origin being fixated at time $t$ is still not known. It is conjectured that this decay is exponential: For every $p \in [0,1]$, there exists $\useconstant{c:fixation_decay}=\useconstant{c:fixation_decay}(p)>0$ such that
\begin{equation}
\PP_{p}\left[ \eta_{s}(0) \neq \lim_{u} \eta_{u}(0), \text{ for some } s \geq t\right] \leq \useconstant{c:fixation_decay}e^{-\useconstant{c:fixation_decay}^{-1}t}.
\end{equation}
As already mentioned, a partial result towards this can be found in~\cite{cns}. We remark that their proof also works for $\ZZ^{2}$, provided $p$ is not close to $\frac{1}{2}$.

\par As a consequence of~\cite{tt}, majority dynamics on $\ZZ^{2}$ fixates. This allows us to define the limiting configuration $\eta_{\infty}$ as the pointwise limit of $\eta_{t}$. This is a random element of $\{0,1\}^{\ZZ^{2}}$ and one can also ask about its percolative properties. If we define
\begin{equation}
p_{c}^{\infty}=\inf\{p \in [0,1]: \PP_{p}[\eta_{\infty} \text{ percolates}]>0\},
\end{equation}
we have
\begin{equation}
\frac{1}{2} \leq p_{c}^{\infty} \leq \lim_{t} p_{c}(t),
\end{equation}
as we shall see. However, it is not known whether equality holds in any of the two estimates. We conjecture this is the case for the latter.
\begin{conjecture} We have $p_{c}^{\infty}=\lim_{t} p_{c}(t)$.
\end{conjecture}

\par A central piece missing for a proof of this conjecture is some form of spacial correlation decay for the limiting configuration. The techniques from~\cite{cns} can be adapted to prove stretched exponential decay if $p > \lim_{t} p_{c}(t)$, but this is also not known for all values of $p \in [0,1]$.

\par Also not known is if the limiting configuration has percolation at criticality. Notice, however, that this is not the case if $p_{c}^{\infty} = \frac{1}{2}$, since percolation never happens at $p=\frac{1}{2}$. This is stated as Theorem~\ref{t:strictly_decreasing_pc} for finite values of $t$, but the same proof applies to the limiting configuration.

\par Finally, regarding the critical percolation function $t \mapsto p_{c}(t)$, we conjecture that it is strictly decreasing. A possible proof strategy would be to understand pivotal sites of crossing events. To do so, one needs to consider sites whose initial opinions determine the existence of a crossing at a given time $t$. A difficulting factor is that we do not have a good control on the collection of sites whose opinions at time $t$ are influenced by a given pivotal site. We know that this set usually grows linearly with time, but it might be the case that, when a site is pivotal, this set is abnormally big and that many things change at time $t$ just by a single change at time zero.

\par One other lattice where this process might be of interest is the hexagonal lattice, where each face has an opinion that is updated according to the majority rule after an exponential random time (equivalently, majority dynamics evolving on the sites of the triangular lattice). We believe that the proof of Theorem~\ref{t:critical_probability} should still hold in this context. Here, the critical probability for percolation at time zero is $\frac{1}{2}$ and we conjecture that the critical percolation function is constant equal to $\frac{1}{2}$. By symmetry considerations, one easily concludes that this function is bounded from below by $\frac{1}{2}$. However, it is not true that infinite clusters are preserved by the dynamics and this implies that there might be times where the function increases and percolation ceases to exist. At the same time, this apparent lack of monotonicity brings other interesting questions that are easily solved for the square lattice. One might ask, for example, about the existence of exceptional times for percolation at $\frac{1}{2}$. As we observe in Section~\ref{sec:percolation}, these times do not exist for the square lattice, but the proof relies on the persistence of percolation, a fact that is not necessarily true in the hexagonal lattice.

\textbf{Structure of the paper.} In Section~\ref{sec:model}, we collect some facts about majority dynamics. These are general facts and should apply to a general family of graphs. Section~\ref{sec:percolation} presents statements regarding percolation in finite time that follows mainly from the choice of $\ZZ^{2}$ as our underlying graph. Section~\ref{sec:criticality} contains the proof of Thereom~\ref{t:critical_probability}, while Theorems~\ref{t:strictly_decreasing_pc} and~\ref{t:continuity} are proved in Section~\ref{sec:time_perc}.

\textbf{Acknowledgments.} This work was supported by the Israel Science Foundation through grant 575/16 and by the German Israeli Foundation through grant I-1363-304.6/2016.

\section{The model}\label{sec:model}
~
\par Two-dimensional majority dynamics is defined precisely as the Markov process with state space $\{0,1\}^{\ZZ^{2}}$ and generator\footnote{For $x,y \in \ZZ^{2}$, we write $x \sim y$ if $||x-y||_{2}=1$.}
\begin{equation}
Lf(\eta)=\sum_{x \in \ZZ^{2}} (\charf{\{\sum_{y \sim x}\eta(y) \geq 3, \eta(x)=0 \}}+\charf{\{\sum_{y \sim x}\eta(y) \leq 1, \eta(x)=1 \}})(f(\eta^{x})-f(\eta))
\end{equation}
where $f$ is any local function and, for $x \in \ZZ^{2}$,
\begin{equation}
\eta^{x}(z)=\left\{\begin{array}{cl}
1-\eta(x),& \mbox{if}\,\,\, z=x,\\
\eta(z),& \mbox{ otherwise.}
\end{array}
\right.
\end{equation}

\par For $p \in [0,1]$, let $\PP_{p}$ denote the distribution of the process when $(\eta_{0}(x))_{x \in \ZZ^{2}}$ has i.i.d. entries with distribution $\ber(p)$. Denote also by $\PP_{p,t}$ the distribution of $\eta_{t}$.

\par The measures $\PP_{p,t}$ are ergodic separately with
respect to horizontal and vertical translations. Besides, these measures are invariant under reflections and orthogonal rotations of $\ZZ^{2}$.

\bigskip

\par This process also has a graphical construction that will be useful to us. Let $\mathscr{P}=(\mathscr{P}_{x})_{x \in \ZZ^{2}}$ be a collection of independent Poisson processes and $(U_{x})_{x \in \ZZ^{2}}$ be a collection of i.i.d. $\unif[0,1]$ random variables. Given $p \in [0,1]$, set
\begin{equation}
\eta_{0}(x)=\charf{\{U_{x} \leq p\}},
\end{equation}
and evolve the process using the clocks $\mathscr{P}$: Whenever $\mathscr{P}_{x}$ rings at time $t$, update $\eta_{t}(x)$ according to the most common opinion among the neighbors of $x$. Ties are solved by observing the opinion of $x$.

\subsection{Correlation decay of $\PP_{p,t}$}\label{subsec:correlation}
~
\par For any fixed time $t > 0$, the opinions on any two close sites are not independent. However, this dependence decays fast as the distance between the two sites increases. In this subsection we prove these claims.

\par Fix $t \geq 0$ and $x \in \ZZ^{2}$. Let $C_{t}(x)$ denote the set of vertices whose opinion is queried in order to determine $\eta_{t}(x)$. Let $B_{n}=[-n,n]^{2}$ and denote by $x+B_{n}$ the translation of $B_{n}$ by $x \in \ZZ^{2}$.

%%%%%
\newconstant{c:large_cone_of_light}
%%%%%

\par Observe that, for any $u \geq 0$, if $m= \lceil 3e^{2}t+u\rceil$, union bounds and Stirling's formula imply
\begin{equation}\label{eq:correlation_decay}
\begin{split}
\PP_{p,t}[C_{t}(x) \cap \{x+ \partial B_{m}\} \neq \emptyset] & \leq \PP_{p,t}\left[\begin{array}{cl} \text{there exists a simple path of size } m \\ \text{starting at } x \text{ such that all clocks} \\ \text{ring before time } t \text{ in order} \end{array}\right]\\
& \leq 4 \cdot 3^{m} \PP_{p,t}[\poisson(t) \geq m] \\
& \leq 4\cdot 3^{m} e^{-t} \sum_{k \geq m} \frac{t^{k}}{k!} \\ & \leq 4 \cdot 3^{m} e^{-t} e^{t} \frac{t^{m+1}}{(m+1)!} \\ & \leq \frac{4t}{\sqrt{2 \pi}} \exp\left\{ m(1+\log 3 + \log t-\log m) \right\} \\ & \leq \useconstant{c:large_cone_of_light}e^{-u},
\end{split}
\end{equation}
for some positive constant $\useconstant{c:large_cone_of_light}=\useconstant{c:large_cone_of_light}(t)$.

\par Fix $x, y \in \ZZ^{2}$ with $\dist=\dist(x,y)>6e^{2}t+1$. From the expression above, we get
\begin{equation}\label{eq:disjoing_cone_of_light}
\PP_{p,t}[C_{t}(x) \cap C_{t}(y) \neq \emptyset] \leq 2\PP_{p,t}[C_{t}(x) \cap \{x+\partial B_{\sfrac{d}{2}}\} \neq \emptyset] \leq 2\useconstant{c:large_cone_of_light}e^{-\sfrac{d}{2}}.
\end{equation}
By possibly increasing the value of $\useconstant{c:large_cone_of_light}$, the estimate above holds true for all $x, y \in \ZZ^{2}$.

%%%%%
\newconstant{c:covariance}
%%%%%

\par If $C_{t}(x) \cap C_{t}(y) = \emptyset$, then $\eta_{t}(x)$ and $\eta_{t}(y)$ are determined by disjoint parts of the initial condition. In particular, directly from~\eqref{eq:disjoing_cone_of_light} we obtain
\begin{equation}
\cov(\eta_{t}(x), \eta_{t}(y)) \leq \useconstant{c:covariance}e^{-\frac{\dist(x,y)}{2}},
\end{equation}
for some $\useconstant{c:covariance}=\useconstant{c:covariance}(t)$.

%%%%%
\newconstant{c:cov_circ}
%%%%%

\par We now collect some consequences of the estimates above that will be used in the rest of the paper. First, if $A$ is an event that depends on the configuration $\eta_{t}(x)$ only on sites inside $[-n,n]^{2}$ and $B$ is an event that depends on the configuration on sites outside $[-2n,2n]^{2}$, then
\begin{equation}\label{eq:correlation_decay_radius}
\begin{split}
\PP_{p,t}[A \cap B] & \leq \PP_{p,t}[A]\PP_{p,t}[B]+ \sum_{||x||=n, ||y||=2n}\PP_{p,t}[C_{t}(x) \cap C_{t}(y) \neq \emptyset] \\ & \leq \PP_{p,t}[A]\PP_{p,t}[B]+\useconstant{c:cov_circ}n^{2}e^{-\sfrac{n}{2}},
\end{split}
\end{equation}
for some $\useconstant{c:cov_circ}=\useconstant{c:cov_circ}(t)>0$.

%%%%%
\newconstant{c:decoupling}
%%%%%

\par We can also obtain the following bound: If $B_{1}$ and $B_{2}$ are two boxes with
\begin{equation}
\dist(B_{1}, B_{2}) \geq 3e^{2}t,
\end{equation}
then, for any events $A$ and $B$ with respective supports in the boxes $B_{1}$ and $B_{2}$, we have
\begin{equation}\label{eq:decoupling}
\begin{split}
\PP_{p,t}[A \cap B] & \leq \PP_{p,t}[A]\PP_{p,t}[B]+ \sum_{x \in \partial B_{1}, y \in \partial B_{2}}\PP_{p,t}[C_{t}(x) \cap C_{t}(y) \neq \emptyset] \\ & \leq \PP_{p,t}[A]\PP_{p,t}[B]+\useconstant{c:decoupling}| \partial B_{1}||\partial B_{2}|e^{-\frac{\dist(B_{1}, B_{2})}{2}}.
\end{split}
\end{equation}

\begin{remark}
All the constants above may depend on $t$ and can be taken to be non-decreasing in the variable $t$. This is a direct consequence of the relation $C_{t}(x) \subset C_{s}(x)$, if $t \leq s$.
\end{remark}

\subsection{Positive association}\label{subsec:fkg}
~
\par A distribution $\mu$ on $\{0,1\}^{S}$ is said positively associated if, for any bounded non-decreasing functions $f,g:\{0,1\}^{S} \to [0,1]$,
\begin{equation}\label{eq:general_fkg}
\EE_{\mu}[fg] \geq \EE_{\mu}[f]\EE_{\mu}[g],
\end{equation}
and we call the estimate above the FKG inequality.

\begin{remark}
Notice that, if~\eqref{eq:general_fkg} holds for any two non-decreasing functions, the same is true for any pair of non-increasing functions. To verify this, it is enough to apply~\eqref{eq:general_fkg} to $1-f$ and $1-g$, where $f$ and $g$ are any two non-increasing functions.
\end{remark}

\par In this section, we prove this property for the measures $\PP_{p,t}$. As an auxiliary result, we use the following.

\begin{teo}[\cite{harris}] Let $(X_{t})_{t \geq 0}$ be a Markov process with space state $\{0,1\}^{S}$, where $S$ is a finite set. Assume that the Markov process satisfies
\begin{enumerate}
\item $|\{x: X_{t-}(x) \neq X_{t}(x)\}| \leq 1$, for all $t \geq 0$;
\item For any non-decreasing function $f:\{0,1\}^{S} \to [0,1]$, $X_{0} \mapsto \EE_{X_{0}}[f(X_{t})]$ is non-decreasing for every $t \geq 0$.
\end{enumerate}
If the initial distribution of $X_{0}$ is positively associated, then the same holds for every fixed finite time.
\end{teo}

\par With the last theorem, we can conclude that the same type of inequality holds for majority dynamics.

\begin{prop}\label{prop:fkg}
If $A$ and $B$ are increasing events, then
\begin{equation}
\PP_{p,t}[A \cap B] \geq \PP_{p,t}[A] \PP_{p,t}[B].
\end{equation}
\end{prop}

\begin{proof}
Consider first the case when $A$ and $B$ have finite support. Recall that, for $n \in \NN$, $B_{n}=[-n,n]^{2}$ and fix $m$ such that $\text{supp}(A) \cup \text{supp}(B) \subset B_{m}$.

Recall that, for $x \in \ZZ^{2}$, $C_{t}(x)$ denotes the set of points in $\ZZ^{2}$ whose opinions are queried when determining $\eta_{t}(x)$ and let $C_{n}$ denote the event that, for some $x \in B_{m}$, $C_{t}(x) \cap B_{n+m}^{c} \neq \emptyset$. In $C_{n}^{c}$, for every $x \in B_{m}$, $\eta_{t}(x)$ is determined by the initial configuration restricted to $B_{n+m}$. We consider $m \leq n$ and denote by $(\eta_{t}^{n})_{t \geq 0}$ a majority dynamics evolving on the finite graph $B_{2n}$. If one uses the same Poisson clocks for both $(\eta_{t}^{n})_{t \geq 0}$ and $(\eta_{t})_{t \geq 0}$, we get
\begin{equation}
[\eta_{t} \in A \cap B] = [\eta_{t}^{n} \in A \cap B] \qquad \text{in } C_{n}^{c},
\end{equation}
and the same is true if we replace $A \cap B$ by $A$ or $B$ in the expression above. This implies
\begin{equation}
\begin{split}
\PP_{p,t}[A \cap B] & \geq \PP[ \eta_{t} \in A \cap B, C_{n}^{c}] \\
& = \PP_{p}[\eta_{t}^{n} \in A \cap B, C_{n}^{c}] \\
& \geq \PP_{p}[\eta_{t}^{n} \in A \cap B]-\PP_{p}[C_{n}] \\
& \geq \PP_{p}[\eta_{t}^{n} \in A] \PP_{p}[\eta_{t}^{n} \in B]-\PP_{p}[C_{n}] \\
& \geq \PP_{p,t}[A]\PP_{p,t}[B]-3\PP_{p}[C_{n}].
\end{split}
\end{equation}

Now, according to Equation~\eqref{eq:correlation_decay}, we have
\begin{equation}
\PP_{p}[C_{n}] \leq \useconstant{c:large_cone_of_light}(2n+1)^{2}e^{-n}.
\end{equation}
By taking the limit $n \to \infty$, we conclude this first part.

The general case follows the usual proof of the FKG inequality (see Theorem~2.4 from~\cite{grimmett}).
\end{proof}

\begin{remark}
We observe that the FKG inequality also holds for the limiting configuration $\eta_{\infty}$. For a proof, one uses bounded convergence theorem for the case when the increasing events have finite support and the usual martingale convergence approach to conclude the general case.
\end{remark}

\section{Majority dynamics percolation}\label{sec:percolation}
~
\par We now start to study percolative properties of majority dynamics. In this section, we present some initial results. These are simple facts that follow mainly from the choice of $\ZZ^{2}$ as our underlying graph.

\par We state the results from~\cite{gandolfi}, that consider general two-dimensional site percolation models.

\begin{teo}[\cite{gandolfi}]\label{t:gandolfi}
Assume $\PP$ is a probability measure on $\{0,1\}^{\ZZ^{2}}$ that satisfies
\begin{enumerate}
\item $\PP$ is invariant under horizontal and vertical translations and axis reflections;
\item $\PP$ is ergodic (separately) under horizontal and vertical translations;
\item $\PP$ is positively associated;
\item Percolation occurs with positive probability under $\PP$.
\end{enumerate}
If all the hypotheses above are satisfied, there exists almost surely exactly one infinite open component. Besides, any finite set of sites is surrounded by an occupied circuit with probability one.
\end{teo}

\par This result has some consequences when it comes to majority dynamics. First of all, for $p=\frac{1}{2}$, symmetry and Theorem~\ref{t:gandolfi} imply that percolation is not possible at any fixed time $t \geq 0$. This gives us that
\begin{equation}
p_{c}(t) \geq \frac{1}{2}, \quad \text{for all } t \geq 0.
\end{equation}

\par One may also ask whether exceptional times exist in this case, i.e., ask if $(\eta_{t})_{t \geq 0}$ percolates at some random time $\tau$. This corresponds to asking if the probability
\begin{equation}
\PP_{\sfrac{1}{2}}[\eta_{t} \text{ percolates, for some } t \geq 0]
\end{equation}
is positive or not.

\par If the probability above is positive, there is a positive probability that this happens before some large time $T$. Since we know that percolation is preserved by the dynamics, there is a positive probability that percolation occurs at time $T$. By ergodicity, this probability is one, contradicting the fact that percolation cannot exist at any fixed time.

\bigskip

\par Theorem~\ref{t:gandolfi} also allows us to prove some percolative properties of the asymptotic configuration $\eta_{\infty}$. Suppose that percolation happens for some time $t \geq 0$. We already know that percolation occurs for all times $s \geq t$. However, our argument does not imply that percolation holds for the limiting configuration. To conclude this, observe that, if we have percolation at time $t$, there exists an infinite open path of vertices $x_{0}, x_{1}, x_{2}, \dots$. At the same time, Theorem~\ref{t:gandolfi} implies that $x_{0}$ is surrounded by an open circuit $y_{0}, y_{1}, \dots, y_{k}=y_{0}$. We have $x_{i}=y_{j}$, for some $i$ and $j$. The collection of open vertices $\{ y_{0}, y_{1}, \dots, y_{k-1} \} \cup \{ x_{i+1}, x_{i+2}, \dots \}$ forms an stable open structure and concludes the proof.

\subsection{Box crossing functions}\label{subsec:bcf}
~
\par We now define the main type of event that we will consider in the rest of the text. We will focus in understanding basic properties of the box crossing probabilities.

\par Given $R \subseteq \ZZ^{2}$ and two disjoint subsets $A$ and $B$ of $R$, define the event
\begin{equation}
\left[A  \overset{R}{\longleftrightarrow} B\right]
\end{equation}
as the existence of an open path contained in $R$ connecting $A$ to $B$. Let also
\begin{equation}
\left[A \longleftrightarrow \infty \right]
\end{equation}
denote the event that there exists an infinite open path starting at some site in $A$.

\par For $m, n \in \NN$, define the event $H(n,m)$ as
\begin{equation}
H(n,m)=\left[\{1\} \times [1,m] \overset{R_{n,m}}{\longleftrightarrow} \{n\} \times [1,m]\right],
\end{equation}
where $R_{n,m}=[1,n] \times [1,m]$, the existence of an open crossing of $R_{n,m}$ connecting its left boundary to its right boundary.

\par We first prove an easy lemma regarding these probabilities.
\begin{lemma}\label{lemma:asymptotic_crossing}
If $\sup_{n} \PP_{p,t}[H(3n,n)]=1$, then $\lim_{n}\PP_{p,t}[H(3n,n)]=1$.
\end{lemma}

\begin{proof}
We investigate the collection of events $H(3n,n)^{c}$. Define the sequence
\begin{equation}
L_{0}=n_{0}, \qquad \text{and} \qquad L_{k+1}=3L_{k},
\end{equation}
where $n_{0}$ will be taken large enough, and consider the probabilities
\begin{equation}\label{eq:prob_seq}
p_{k}=\PP_{p,t}[H(3L_{k},L_{k})^{c}].
\end{equation}

Suppose we are in the event $H(3L_{k+1},L_{k+1})^{c}$, and consider the collection of rectangles
\begin{equation}\label{eq:rectangles}
\begin{split}
[2xL_{k}+1, (2x+3)L_{k}] \times [1,L_{k}], & \qquad x \in \{0,1,2,3\}, \\
[2xL_{k}+1, (2x+1)L_{k}] \times [-2L_{k}+1, L_{k}], & \qquad x \in \{1,2,3\}.
\end{split}
\end{equation}

Observe that, if all the seven rectangles listed above are crossed in the hard direction, then we can concatenate these paths and find a hard crossing of $[1,3L_{k+1}] \times [1,L_{k+1}]$. This implies that, in $H(3L_{k+1},L_{k+1})^{c}$, at least one of the rectangles in~\eqref{eq:rectangles} is not crossed in the hard direction.

Similarly, we can construct another set of seven rectangles along the upper boundary of $[1,3L_{k+1}] \times [1,L_{k+1}]$ with the same property.

The minimum distance between rectangles from the upper and lower chain is $L_{k}$. Union bound and Equation~\eqref{eq:decoupling} imply
\begin{equation}
p_{k+1} \leq 49(p_{k}^{2}+64\useconstant{c:decoupling}L_{k}^{2}e^{-\sfrac{L_{k}}{2}}),
\end{equation}
whenever $L_{k} \geq 3e^{2}t$.

Choose $n_{0} \geq 3e^{2}t$ large so that $p_{0} \leq (4 \cdot 49)^{-1}$ (see~\eqref{eq:prob_seq}) and
\begin{equation}
49^{2} \cdot 512\useconstant{c:decoupling} n^{3}e^{-\sfrac{n}{2}} \leq \frac{1}{2}, \qquad \qquad \text{for all } n \geq n_{0}.
\end{equation}

Suppose that $p_{k} \leq (4\cdot 49)^{-1}2^{-k}$ and observe that
\begin{equation}
\begin{split}
4\cdot 49 \cdot2^{k+1}p_{k+1} & \leq 4\cdot 49^{2}\cdot2^{k+1}(p_{k}^{2}+64\useconstant{c:decoupling}L_{k}^{2}e^{-\sfrac{L_{k}}{2}}) \\
& \leq 4\cdot 49^{2}\cdot2^{k+1}\left(\frac{1}{16\cdot 49^{2}}\frac{1}{2^{2k}}+64\useconstant{c:decoupling}L_{k}^{2}e^{-\sfrac{L_{k}}{2}}\right) \\
& \leq \frac{1}{2} + 49^{2}\cdot 512 \useconstant{c:decoupling}L_{k}^{3}e^{-\sfrac{L_{k}}{2}} \leq 1.
\end{split}
\end{equation}

This implies $p_{k} \leq  (4\cdot 49)^{-1}2^{-k}$, for all $k \geq 0$ and that $\PP_{p,t}[H(3n,n)^{c}]$ converges to zero along the subsequence $L_{k}$.

Now, given $n \in \NN$ large, consider $k$ with
\begin{equation}
L_{k} \leq n < L_{k+1}.
\end{equation}
With a concatenation using the rectangles described in \eqref{eq:rectangles}, we get
\begin{equation}
\PP_{p,t}[H(3n,n)^{c}] \leq 7\PP_{p,t}[H(3L_{k},L_{k})^{c}]
\end{equation}
and conclude the proof.
\end{proof}

\par The main result regarding crossing probabilities is Russo-Seymour-Welsh theory.
\begin{teo}\label{t:rsw}
If $\inf_{n} \PP_{p,t}[H(n,n)]>0$, then, for all $\lambda>0$, we also have $\inf_{n} \PP_{p,t}[H(\lambda n,n)]>0$.
\end{teo}

\par The proof follows the same steps from~\cite{tassion} and we omit it here.

\section{Criticality regime of $\PP_{p,t}$}\label{sec:criticality}
~
\par This section is devoted to the proof of Theorem~\ref{t:critical_probability}. The proof is based on the understanding of the supercritical phase. We first consider the values of $p$ and $t$ where percolation occurs and study crossing functions in this case. Then we conclude the proof of Theorem~\ref{t:critical_probability}.

\par The first theorem states that crossings are very likely if percolation occurs.
\begin{teo}\label{t:box_crossing}
If $\PP_{p,t}[\eta \text{ percolates}]>0$, then, for every $\lambda>0$,
\begin{equation}
\sup_{n}\PP_{p,t}[H(\lambda n,n)]=1.
\end{equation}
Also, $\lim_{n}\PP_{p,t}[H(\lambda n,n)]=1$.
\end{teo}

\par The proof of this theorem is split into several lemmas. The first lemma considers crossings of squares.
\begin{lemma}
Assume that $\PP_{p,t}[\eta \text{ percolates}]>0$. Then
\begin{equation}
\inf_{n} \PP_{p,t}[H(n,n)] >0.
\end{equation}
\end{lemma}

\begin{proof}
Notice initially that
\begin{equation}
\PP_{p,t}\left[(0,0)\longleftrightarrow \infty \right]>0.
\end{equation}
Indeed, assume this was not the case and that the probability above equals zero. By translation invariance, the same would apply for all points $(x,y) \in \ZZ^{2}$. Union bound would give
\begin{equation}
\PP_{p,t}[\eta \text{ percolates}] \leq \sum_{(x,y) \in \ZZ^{2}} \PP_{p,t}\left[(x,y)\longleftrightarrow \infty \right] =0,
\end{equation}
contradicting our hypothesis that this probability is positive.

For $n \in \NN$, we consider the event
\begin{equation}
A_{n}^{right}= \left[(0,0)\overset{B_{n}}{\longleftrightarrow} \{n\} \times [-n,n]\right],
\end{equation}
where $B_{n}=[-n,n] \times [-n,n]$. By rotation invariance, we have
\begin{equation}
\PP_{p,t}\left[(0,0) \longleftrightarrow \infty \right] \leq 4\PP_{p,t}[A_{n}^{right}].
\end{equation}
Now, FKG inequality and symmetry imply that
\begin{equation}
\PP_{p,t}[H(2n+1,2n+1)] \geq \PP_{p,t}[A_{n}^{right}]^{2} \geq \frac{1}{16}\PP_{p,t}\left[(0,0)\longleftrightarrow \infty \right]^{2},
\end{equation}
concluding the lemma.
\end{proof}

\par For $m \leq n$, we consider events of the form
\begin{equation}
\cir(m,n) = \left\{\begin{array}{cl}
\text{there exists an open circuit} \\ \text{ surrounding } [-m,m]^{2} \\ \text{and contained in } [-n,n]^{2}
\end{array}\right\},
\end{equation}

%%%%%
\newconstant{c:cir_1}
\newconstant{c:cir_2}
%%%%%
\begin{prop}\label{prop:circ}
Assume $\inf_{n}\PP_{p,t}[H(n,n)]>0$. There exist positive constants $\useconstant{c:cir_1}$ and $\useconstant{c:cir_2}$ such that, for any $m \leq \frac{n}{9}$,
\begin{equation}
\PP_{p,t}[\cir(m,n)^{c}] \leq \left(\frac{n}{m}\right)^{-\useconstant{c:cir_1}}+\useconstant{c:cir_2}e^{-\useconstant{c:cir_2}m}.
\end{equation}
\end{prop}

\begin{proof}
Let $J=\{j \geq 0: m \leq 3^{j}m \leq 3^{j+1}m \leq n\}$ and observe that $|J| \geq \lfloor \frac{1}{\log3}\log\left(\frac{n}{m}\right) \rfloor \geq 2$.

By Theorem~\ref{t:rsw}, we have
\begin{equation}
\inf_{n}\PP_{p,t}[H(3n,n)]=c>0,
\end{equation}
and this implies, as a simple consequence of FKG inequality,
\begin{equation}
\PP_{p,t}[\cir(3^{j}m, 3^{j+1}m)] \geq c^{4}.
\end{equation}

Observe now that
\begin{equation}
\cir(m,n)^{c} \subseteq \bigcap_{j \in J}\cir(3^{j}m, 3^{j+1}m)^{c}.
\end{equation}
Consider the events $A=\cir(m, 3m)$ and $B=\cap_{j \in J, j \geq 2} \cir(3^{j}m, 3^{j+1}m)$. While $A$ depends on the configuration only inside $[-3m,3m]^{2}$, $B$ is determined by the configuration outside $[-9m,9m]^{2}$. We can then apply~\eqref{eq:correlation_decay_radius} to conclude that
\begin{equation}
\begin{split}
\PP_{p,t}[\cir(m,n)^{c}] & \leq \PP_{p,t}\left[\bigcap_{j \in J}\cir(3^{j}m, 3^{j+1}m)^{c}\right] \\
& \leq \PP_{p,t}[\cir(m, 3m)]\PP_{p,t}[B]+\useconstant{c:cov_circ}m^{2}e^{-\sfrac{m}{2}}\\
& \leq (1-c^{4})\PP_{p,t}[B]+\useconstant{c:cov_circ}m^{2}e^{-\sfrac{m}{2}}.
\end{split}
\end{equation}

We now proceed inductively to bound the probability of $B$ in the same way. Finally, we obtain
\begin{equation}
\begin{split}
\PP_{p,t}[\cir(m,n)^{c}] & \leq (1-c^{4})^{\left\lfloor\frac{|J|}{2}\right\rfloor}+\left(\sum_{j=0}^{\infty}(1-c^{4})^{j}\right)\useconstant{c:cov_circ}m^{2}e^{-\sfrac{m}{2}} \\
& \leq \left(\frac{n}{m}\right)^{-\useconstant{c:cir_1}}+\useconstant{c:cir_2}e^{-\useconstant{c:cir_2}m},
\end{split}
\end{equation}
by a suitable choice of constants. This concludes the proof.
\end{proof}

%%%%%
\newconstant{c:cir_3}
%%%%%
\begin{cor}\label{cor:circ}
Assume $\inf_{n}\PP_{p,t}[H(n,n)]>0$. There exist positive constants $\useconstant{c:cir_3}$ and $\gamma$ such that, for all $n \geq 0$
\begin{equation}
\PP_{p,t}[\cir(\sqrt{n},n)^{c}] \leq \useconstant{c:cir_3}n^{-\gamma}.
\end{equation}
\end{cor}

\par Now we prove that, if percolation occurs, the probability that a square is crossed actually converges to one as the size of the square grows. This implies Theorem~\ref{t:box_crossing} for values $\lambda \leq 1$.
\begin{lemma}
If $\PP_{p,t}[\eta \text{ percolates}]>0$, then
\begin{equation}
\lim_{n}\PP_{p,t}[H(n,n)]=1.
\end{equation}
\end{lemma}

\begin{proof}
Let $B_{n}=[-n,n]^{2}$ and denote by $A(n)$ the event that the annulus $B_{n} \setminus B_{\lfloor\sqrt{n}\rfloor}$ has an open crossing connecting the inner boundary to the outer boundary.

The event $[\eta \text{ percolates}]$ is translation invariant and, by hypothesis, it has positive probability. This implies that $\PP_{p,t}[\eta \text{ percolates}]=1$, and we obtain $\lim_{n} \PP_{p,t}[A(n)^{c}]=0$, since the probability that there exists some vertex inside $B_{\lfloor\sqrt{n}\rfloor}$ that belongs to the infinite cluster converges to one as $n$ grows, and this last event is contained in $A(n)$.

Define $A^{right}(n)$ as the event where the open crossing contained in the annulus $B_{n} \setminus B_{\lfloor\sqrt{n}\rfloor}$ connects the inner boundary to the right outer boundary $\{n\} \times [-n,n]$. By FKG and rotation invariance, we get
\begin{equation}
\PP_{p,t}[A(n)^{c}] \geq \PP_{p,t}[A^{right}(n)^{c}]^{4},
\end{equation}
so that 
\begin{equation}\label{eq:limit_1}
\lim_{n}\PP_{p,t}[A^{right}(n)]=1,
\end{equation}
since $\lim_{n}\PP_{p,t}[A(n)^{c}]=0$.

Now, in order to obtain a crossing connecting the left boundary to the right boundary of $[-n,n]^{2}$, it is enough that $A^{right}(n)$, $A^{left}(n)$ (defined analogously, but connecting to the left outer boundary) and $\cir(\sqrt{n}, n)$ hold. By FKG inequality, we obtain
\begin{equation}
\PP_{p,t}[H(2n+1, 2n+1)] \geq \PP_{p,t}[A^{right}(n)]^{2}\PP_{p,t}[\cir(\sqrt{n}, n)].
\end{equation}
Finally, Corollary~\ref{cor:circ} and~\eqref{eq:limit_1} imply that the right-hand side of the equation above converges to one as $n$ grows, concluding the proof.
\end{proof}

\par We are now ready to conclude the proof of Theorem~\ref{t:box_crossing}.

\begin{proof}[Proof of Theorem~\ref{t:box_crossing}.]
Fix $k \in \NN$ even and, for $n \geq k$, partition the interval $[1,n]$ into $k$ subintervals $I_{n}^{k,i}=(i\frac{n}{k}, (i+1)\frac{n}{k}]$, for $i=0, \dots, k-1$. Here we ignore divisibility issues and assume that all the intervals $I_{n}^{k,i}$ have the same size.

Define the events
\begin{equation}
A_{n}^{k,i}=\left[\{1\} \times [1,n] \overset{R_{n,n}}{\longleftrightarrow} \{n\} \times I_{n}^{k,i}\right],
\end{equation}
where $R_{n,n}=[1,n]^{2}$, see Figure~\ref{fig:events}. We claim that $\lim_{n}\PP_{p,t}[A_{n}^{k,\sfrac{k}{2}}]=1$.

To prove this, first notice that, by FKG inequality,
\begin{equation}
\prod_{i=0}^{k-1}(1-\PP_{p,t}[A_{n}^{k,i}]) = \prod_{i=0}^{k-1}\PP_{p,t}[(A_{n}^{k,i})^{c}] \leq \PP_{p,t}[H(n,n)^{c}] \to 0.
\end{equation}
Hence,
\begin{equation}
\lim_{n} \left(\max_{i}\PP_{p,t}[A_{n}^{k,i}]\right)=1.
\end{equation}
For $k=2$, we have $\PP_{p,t}[A_{n}^{2,0}]= \PP_{p,t}[A_{n}^{2,1}]$ and the limit above is reduced to $\lim_{n} \PP_{p,t}[A_{n}^{2,1}]=1$.

Set
\begin{equation}
\bar{A}_{n}^{k}=\left[\{1\} \times [1,2n] \overset{R_{n,2n}}{\longleftrightarrow} \{n\} \times I_{n}^{k,\sfrac{k}{2}}\right],
\end{equation}
where $R_{n,2n}=[1,n] \times [1,2n]$ (see Figure~\ref{fig:events}). Let us now verify that, for each $k \geq 2$, $\lim_{n}\PP_{p,t}[\bar{A}_{n}^{k}]=1$.

\begin{figure}\label{fig:events}
\begin{center}
\begin{tikzpicture}[scale=0.6]

\node[below] at (1,-1.05){$n$};
\node[left] at (-0.05,0){$n$};

\draw[thick] (0,-1) rectangle (2,1);
\draw[thick](1.9,0.25)--(2.1,0.25);
\draw[thick](1.9,0)--(2.1,0);

\draw[thick] (0, 0) .. controls (0.7, -1) and (1, 1) .. (2, 0.13);

\node[below] at (6,-1.05){$n$};
\node[left] at (4.95,1){$2n$};

\draw[thick] (5,-1) rectangle (7,3);
\draw[thick](6.9,0.25)--(7.1,0.25);
\draw[thick](6.9,0)--(7.1,0);

\draw[thick] (5, 1.5) .. controls (5.7, -1.5) and (6, 1) .. (7, 0.13);
\end{tikzpicture}
\caption{The events $A_{n}^{k, \sfrac{k}{2}}$ and $\bar{A}^{k}_{n}$.}
\end{center}
\end{figure}
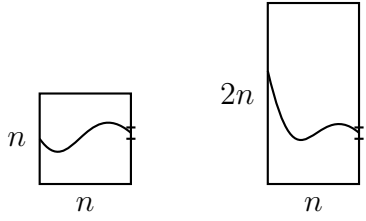

Let $A_{n}^{k,i}(x,y)$ denote the event $A_{n}^{k,i}$ translated by the vector $(x,y)$, i.e.,
\begin{equation}
A_{n}^{k,i}=\left[\{x+1\} \times (y+[1,n]) \overset{(x,y)+R_{n,n}}{\longleftrightarrow} \{x+n\} \times (y+I_{n}^{k,i})\right],
\end{equation}
then
\begin{equation}\label{eq:inclusion}
\bigcup_{i \leq \frac{k}{2}-1} A_{n}^{k,i}\left(0,\left(\frac{k}{2}-i\right)\frac{n}{k}\right) \subseteq \bar{A}_{n}^{k}.
\end{equation}
See Figure~\ref{fig:inclusion_intersection} for an example of the inclusion above in the case $i=0$.

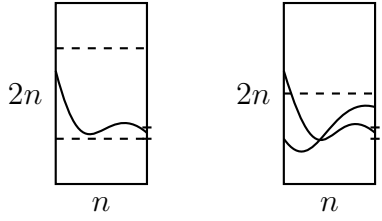
\begin{figure}\label{fig:inclusion_intersection}
\begin{center}
\begin{tikzpicture}[scale=0.6]

\node[below] at (1,-1.05){$n$};
\node[left] at (-0.05,1){$2n$};

\draw[thick, dashed] (0,2) -- (2,2);
\draw[thick, dashed] (0,0) -- (2,0);
\draw[thick](1.9,0.25)--(2.1,0.25);
\draw[thick](1.9,0)--(2.1,0);
\draw[thick] (0,-1) rectangle (2,3);

\draw[thick] (0, 1.5) .. controls (0.7, -1.2) and (1, 1) .. (2, 0.13);

\node[below] at (6,-1.05){$n$};
\node[left] at (4.95,1){$2n$};

\draw[thick, dashed] (5,1) -- (7,1);
\draw[thick](6.9,0.25)--(7.1,0.25);
\draw[thick](6.9,0)--(7.1,0);
\draw[thick] (5,-1) rectangle (7,3);

\draw[thick] (5, 0) .. controls (5.7, -1) and (6, 1) .. (7, 0.7);
\draw[thick] (5, 1.5) .. controls (5.9, -1.5) and (6, 1) .. (7, 0.13);

\end{tikzpicture}
\caption{On the left, the inclusion $A_{n}^{k,0}\left(0,\frac{n}{2}\right) \subseteq \bar{A}_{n}^{k}$, and, on the right, the intersection $\bar{A}_{n}^{k} \cap A_{n}^{2,1}$.}
\end{center}
\end{figure}

The inclusion~\eqref{eq:inclusion} implies
\begin{equation}
\begin{split}
\lim_{n}\PP_{p,t}[\bar{A}_{n}^{k}] & \geq \lim_{n}\PP_{p,t}\left[\bigcup_{i \leq \frac{k}{2}-1} A_{n}^{k,i}\left(0,\left(\frac{k}{2}-i\right)\frac{n}{k}\right)\right] \\
& \geq \lim_{n} \max_{i \leq \sfrac{k}{2}-1}\PP_{p,t}[A_{n}^{k,i}]=1.
\end{split}
\end{equation}

Now, simply observe that, using a concatenation argument (see Figure~\ref{fig:inclusion_intersection}), we can conclude that $A_{n}^{k,\sfrac{k}{2}} \supseteq \bar{A}_{n}^{k} \cap A_{n}^{2,1}$. This implies the claim, since the probabilities of both events in the intersection converge to one.

To conclude the theorem it is enough to consider $\lambda =2$. Fix $\epsilon>0$, and choose $k \in 2\NN$ large such that
\begin{equation}
k^{-\useconstant{c:cir_1}} \leq \frac{\epsilon}{2},
\end{equation}
where $\useconstant{c:cir_1}$ is given by Proposition~\ref{prop:circ}. Choose $n_{0} \geq 9k$ such that, for all $n \geq n_{0}$,
\begin{equation}
\PP_{p,t}[A_{n}^{k,\sfrac{k}{2}}] \geq 1-\epsilon \quad \text{and} \quad 
\useconstant{c:cir_2}e^{-\useconstant{c:cir_2}\frac{n}{k}} \leq \frac{\epsilon}{2},
\end{equation}
where $\useconstant{c:cir_2}$ is given by Proposition~\ref{prop:circ}.

For this choice of $k$ and $n_{0}$, Proposition~\ref{prop:circ} implies, for $n \geq n_{0}$,
\begin{equation}
\PP_{p,t}\left[\cir\left(\frac{n}{k},n\right)^{c}\right] \leq \epsilon.
\end{equation}
Finally, by considering suitable translations and reflections of the events $\cir\left(\frac{n}{k},n\right)$ and $A_{n}^{k,\sfrac{k}{2}}$, that we represent with a drawing, we can apply FKG inequality to obtain
\begin{equation}
\begin{split}
\PP_{p,t}[H(2n,n)] & \geq \PP_{p,t} \left[ \drawing \right] \\
& \geq \PP_{p,t}[A_{n}^{k,\sfrac{k}{2}}]^{2}\PP_{p,t}\left[\cir\left(\frac{n}{k},n\right)\right] \geq (1-\epsilon)^{3},
\end{split}
\end{equation}
for any $n \geq n_{0}$. Since the choice of $\epsilon$ is arbitrary, the result follows.
\end{proof}

Our next lemma gives a sufficient condition for percolation. The proof of this criterion follows the same steps of Lemma~\ref{lemma:asymptotic_crossing}.
\begin{lemma}\label{lemma:percolation_condition}
Given $T$, there exist $\epsilon>0$ and $n_{0} \in \NN$ such that if
\begin{equation}
\PP_{p,t}[H(4n,n)] \geq 1-\epsilon,
\end{equation}
for some $t \in [0,T], p \in [0,1]$ and $n \geq n_{0}$, then
\begin{equation}
\PP_{p,t}[\eta \text{ percolates}]>0.
\end{equation}
\end{lemma}

\begin{proof}
Define $\epsilon=(4\cdot 14^{2})^{-1}$, and let $n_{0} \geq 3e^{2}T$ be such that
\begin{equation}
100\useconstant{c:decoupling}n^{2}e^{-n} \leq \frac{n^{-1}}{8 \cdot 14^{4}},
\end{equation}
for all $n \geq n_{0}$, where $\useconstant{c:decoupling}=\useconstant{c:decoupling}(T)$ is the constant in~\eqref{eq:decoupling}.

Assume that $\PP_{p,t}[H(4n,n)] \geq 1-\epsilon$, for some $n \geq n_{0}$, define the sequence
\begin{equation}
L_{0}=n, \qquad \text{and} \qquad L_{k+1}=4L_{k},
\end{equation}
and consider the probabilities
\begin{equation}
p_{k}=\PP_{p,t}[H(4L_{k},L_{k})^{c}].
\end{equation}

Suppose we are in the event $H(4L_{k+1},L_{k+1})^{c}$, and consider the collection of rectangles
\begin{equation}\label{eq:rectangles_2}
\begin{split}
[2xL_{k}+1, (2x+4)L_{k}] \times [1,L_{k}], & \qquad x \in \{0,1,2,3,4,5,6,7\}, \\
[2xL_{k}+1, (2x+1)L_{k}] \times [-3L_{k}+1, L_{k}], & \qquad x \in \{1,2,3,4,5,6\}.
\end{split}
\end{equation}

If all the fourteen rectangles above are crossed in the long direction, then we can concatenate these paths and find a long crossing of $[1,4L_{k+1}] \times [1,L_{k+1}]$.

We conclude that, in $H(4L_{k},L_{k})^{c}$, at least one of the rectangles in~\eqref{eq:rectangles_2} is not crossed in the long direction. Similarly, another set of fourteen rectangles can be constructed along the upper boundary of $[1,4L_{k+1}] \times [1,L_{k+1}]$ with the same property.

The minimum distance between rectangles from the upper and lower chain is $2L_{k}$. Union bound and Equation~\eqref{eq:decoupling} imply
\begin{equation}
p_{k+1} \leq 14^{2}(p_{k}^{2}+100\useconstant{c:decoupling}L_{k}^{2}e^{-L_{k}}).
\end{equation}

By assumption, $p_{0} \leq (4\cdot 14^{2})^{-1}$. Suppose, by induction, that $p_{k} \leq (4\cdot 14^{2})^{-1}2^{-k}$ and estimate
\begin{equation}
\begin{split}
4\cdot 14^{2} \cdot2^{k+1}p_{k+1} & \leq 4\cdot 14^{4}\cdot2^{k+1}(p_{k}^{2}+100\useconstant{c:decoupling}L_{k}^{2}e^{-L_{k}}) \\
& \leq 4\cdot 14^{4}\cdot2^{k+1}\left(\frac{1}{16\cdot 14^{4}}\frac{1}{2^{2k}}+100\useconstant{c:decoupling}L_{k}^{2}e^{-L_{k}}\right) \\
& \leq 4\cdot 14^{4}\cdot2^{k+1}\left(\frac{1}{16\cdot 14^{4}}\frac{1}{2^{2k}}+\frac{L_{k}^{-1}}{8\cdot 14^{4}}\right) \\
& \leq \frac{1}{2}+2^{k+1}\frac{L_{k}^{-1}}{2} \leq 1.
\end{split}
\end{equation}
This implies $p_{k} \leq  (4\cdot 14^{2})^{-1}2^{-k}$, for all $k \geq 0$.

Let us now verify that $\PP_{p,t}$ percolates. Consider, for each $k$, the rectangles in~\eqref{eq:rectangles_2}. Since
\begin{equation}
\sum_{k}14p_{k}< \infty,
\end{equation}
by Borel-Cantelli Lemma, there exists $k_{0}$ such that, for all $k \geq k_{0}$, every rectangle is crossed  in the hard direction by an open path. Concatenating these paths allows us to construct an infinite open path and concludes the proof of the lemma.
\end{proof}

\par Our next proposition says that set $\mathscr{P}$ introduced in~\eqref{eq:set_P} is an open subset of $[0,1] \times [0, \infty)$.

\begin{prop}\label{prop:percolation_parameters}
If $\PP_{p,t}[\eta \text{ percolates}]>0$, then there exists $\delta>0$ such that
\begin{equation}
\PP_{p-\delta,t-\delta}[\eta \text{ percolates}]>0.
\end{equation}
\end{prop}

\begin{proof}
Let $\epsilon>0$ and $n_{0} \in \NN$ be the values given by Lemma~\ref{lemma:percolation_condition} if we set $T=t$.

Use Theorem~\ref{t:box_crossing} to find $n \geq n_{0}$ such that
\begin{equation}
\PP_{p,t}[H(4n, n)^{c}]\leq \frac{\epsilon}{3}.
\end{equation}

Now, if $\delta>0$ is small enough, we have
\begin{equation}
\begin{split}
\PP_{p,t-\delta}[H(4n, n)^{c}] & \leq \PP_{p,t}[H(4n, n)^{c}]+\PP_{p}\left[\begin{array}{cl} \eta_{t-\delta}(x) \neq \eta_{t}(x), \\  \text{for some } x \in [1,4n] \times [1,n] \end{array} \right] \\
& \leq \frac{\epsilon}{3}+\frac{\epsilon}{3}=\frac{2\epsilon}{3}.
\end{split}
\end{equation}

Finally, observe that $p \mapsto \PP_{p,t-\delta}[H(4n, n)^{c}]$ is a continuous function. This implies that, for $\tilde{\delta}>0$ small, we have
\begin{equation}
\PP_{p-\tilde{\delta},t-\delta}[H(4n, n)^{c}] \leq \epsilon.
\end{equation}

An application of Lemma~\ref{lemma:percolation_condition} implies the result.
\end{proof}

\par We are now ready to conclude the proof of Theorem~\ref{t:critical_probability}.
\begin{proof}[Proof of Theorem~\ref{t:critical_probability}]

Assume that, for some $t \geq 0$, we have
\begin{equation}
\PP_{p_{c}(t),t}[\eta \text{ percolates}]>0.
\end{equation}

Proposition~\ref{prop:percolation_parameters} implies that
\begin{equation}
\PP_{p_{c}(t)-\delta,t-\delta}[\eta \text{ percolates}]>0.
\end{equation}
Moreover, since percolation is preserved by the dynamics, the same is true for the measure $\PP_{p_{c}(t)-\delta,t}$, a contradiction with the definition of $p_{c}(t)$.
\end{proof}

\section{Critical percolation as a function of $t$}\label{sec:time_perc}
~
\par We now focus on how $p_{c}(t)$ varies with $t$. We first prove Theorem~\ref{t:continuity}, that says $p_{c}(t)$ is a continuous function of $t$. Then, we present the proof of Theorem~\ref{t:strictly_decreasing_pc}.

\begin{proof}[Proof of Theorem~\ref{t:continuity}.]
Since $t \mapsto p_{c}(t)$ is non-increasing, the lateral limits
\begin{equation}
p_{c}(t+)=\lim_{s \to t^{+}} p_{c}(s) \quad \text{and} \quad p_{c}(t-)=\lim_{s \to t^{-}} p_{c}(s)
\end{equation}
exist and are finite for all $t \geq 0$. To conclude the proof, it suffices to verify that they coincide with $p_{c}(t)$.

We first check that $p_{c}(t-)=p_{c}(t)$. Assume, by contradiction, that $p_{c}(t-)>p_{c}(t)$ and choose $p \in (p_{c}(t), p_{c}(t-))$. We apply Proposition~\ref{prop:percolation_parameters} and conclude that, since
\begin{equation}
\PP_{p,t}[\eta \text{ percolates}]>0,
\end{equation}
we have
\begin{equation}
\PP_{p,t-\delta}[\eta \text{ percolates}]>0,
\end{equation}
contradicting the fact that $p < p_{c}(t-) \leq p_{c}(t-\delta)$.

We now verify that $p_{c}(t+)=p_{c}(t)$. We construct a coupling $(\xi, \bar{\xi})$ with respective marginal distributions $\PP_{p+\delta,t+\delta}$ and $\PP_{p+\delta+\delta',t}$, with $\delta'=(1-p-\delta)(1-e^{-\delta})>0$, such that $ \xi \leq \bar{\xi}$.

Let $\mathscr{P}=(\mathscr{P}_{x})_{x \in \ZZ^{2}}$ be a collection of independent Poisson clocks with rate one and $\xi_{0}$ be an initial configuration with density $p+\delta$. The configuration $\xi$ is obtained by running the graphical construction presented in Section~\ref{sec:model} with $\xi_{0}$ and $\mathscr{P}$ up to time $t+\delta$. As for $\bar{\xi}$, we begin with $\xi_{0}$, and between times zero and $\delta$, when a clock $\mathscr{P}_{x}$ rings, we set the entry of $x$ equal to 1. From time $\delta$ up to $t+\delta$, we simply perform majority dynamics. It is easy to verify that the coupling satisfies all the required properties.

Fix $p=p_{c}(t+)$ and use the coupling above. For $\delta>0$, we have $p+\delta > p_{c}(t+\delta)$, and hence
\begin{equation}
\begin{split}
\PP_{p+\delta+\delta',t}[\eta \text{ percolates}] & = \PP_{coupling}[\bar{\xi} \text{ percolates}] \\ 
& \geq \PP_{coupling}[\xi \text{ percolates}] \\
& = \PP_{p+\delta,t+\delta}[\eta \text{ percolates}]>0.
\end{split}
\end{equation}
This implies 
\begin{equation}
p_{c}(t) \leq p_{c}(t+)+\delta+\delta', \quad \text{for all } \delta>0.
\end{equation}
Taking the limit $\delta \to 0$ yields the result.
\end{proof}

\par The remaining of the section is devoted to the proof of Theorem~\ref{t:strictly_decreasing_pc}.

\begin{proof}[Proof of Theorem~\ref{t:strictly_decreasing_pc}.]
We use Aizemann-Grimmett Theorem (Theorem 1 of~\cite{ag} and Theorem 2 of~\cite{bbr}) to obtain a strict inequality. For every $t >0$, we will define an essential enhancement.

First, we partition the sites of $\ZZ^{2}$ in two sets
\begin{equation}\label{eq:even_sites}
A=\left\{(x,y) \in \ZZ^{2}: x+y \text{ is even} \right\} \quad \text{and} \quad B=\ZZ^{2} \setminus A.
\end{equation}

We will define an enhancement only for the sites in $A$ (see Remark~\ref{remark:enhancement} as to why the enhancement theorem applies in this case). According to the graphical construction in Section~\ref{sec:model}, each site of $\ZZ^{2}$ has a Poisson clock associated to it. The distribution of the first time a clock rings in a given site $x \in \ZZ^{2}$ is a random variable $T_{x} \sim \expo(1)$. For every site $x \in B$, we write $T_{x}$ as the sum of four i.i.d. random variables (this is possible since the exponential distribution is infinitely divisible) and associate each one of these variables to a neighboring site of $x$. Each site $y \in A$ will have four random variables associated to it, one from each of its neighboring sites in $B$. These random variables together with $T_{y}$ will determine whether the enhancement at $y$ is activated or not: If $T_{y}<t$ and all the four random variables associated to $y$ are larger than $t$ the enhancement is activated at $y$. When the enhancement is performed at $y$, we change the configuration at $y$ to a $1$ if in the original configuration at least three of its neighbors are a $1$.

Aizemann-Grimmett Theorem implies that
\begin{equation}
p_{c}^{enh} < p_{c}^{site}.
\end{equation}
Hence, it is enough to verify that
\begin{equation}\label{eq:pt_penh}
p_{c}(t) \leq p_{c}^{enh}.
\end{equation}

Fix $p \in (p_{c}^{enh}, p_{c}^{site})$ and let $\{x_{0}, x_{1}, \dots \}$ be an infinite open path for the enhanced configuration. Let us prove that, for some $j \geq 0$, $\{x_{j}, x_{j+1}, \dots \}$ is an open collection of sites for $\eta_{t}$.

We may partition the infinite open path in chains $\{x_{i}, x_{i+1}, \dots, x_{i+k}\}$, $k \geq 0$, in which no site was enhanced and such that both $x_{i-1}$ and $x_{i+k+1}$ have their enhancements performed. We choose $j$ as the first positive index whose enhancement is performed. Each of the chains $\{x_{i}, x_{i+1}, \dots, x_{i+k}\}$ cannot be destroyed before time $t$, since the first and last sites of the chain have a first clock ring only after time $t$ and each site in the middle cannot change opinion. As for the enhanced sites, they have a clock ring before time $t$, all the neighboring sites have the first clock ring after time $t$ and they can change into a one. This implies that every enhanced site is open at time $t$. In particular, $p_{c}(t) \leq p$ and hence~\eqref{eq:pt_penh}, concluding the proof.
\end{proof}

\begin{remark}\label{remark:enhancement}
The enhancement used here does not fit the exact hypotheses from Theorem 2 of~\cite{bbr}. The only difference is that, while in~\cite{bbr} the enhancement is performed on all sites, here we only use sites in $A$ (see~\eqref{eq:even_sites}). It is necessary to verify that the proof given in~\cite{bbr} still holds in this case. The central argument relies on the fact that one can locally modify the configuration to pass from $p$-pivotal to enhancement-pivotal sites. In our case, this argument works for sites in $A$. When considering a $p$-pivotal site in $B$, using the same argument with minor modifications, one can verify that it is possible to locally modify the configuration so that one of the neighbors of the $p$-pivotal site turns into an enhancement-pivotal site. With this in hands, the proof of the enhancement theorem follows the same steps of~\cite{bbr}.
\end{remark}

\bibliographystyle{plain}
\bibliography{mybib}

\begin{thebibliography}{10}

\bibitem{ag}
Michael Aizenman and Geoffrey Grimmett.
\newblock Strict monotonicity for critical points in percolation and
  ferromagnetic models.
\newblock {\em Journal of Statistical Physics}, 63(5-6):817--835, 1991.

\bibitem{arratia}
Richard Arratia.
\newblock Site recurrence for annihilating random walks on $\mathbb{Z}_{d}$.
\newblock {\em The Annals of Probability}, 11(3):706--713, 1983.

\bibitem{bbr}
Paul Balister, B{\'e}la Bollob{\'a}s, and Oliver Riordan.
\newblock Essential enhancements revisited.
\newblock {\em arXiv preprint arXiv:1402.0834}, 2014.

\bibitem{bcott}
Itai Benjamini, Siu-On Chan, Ryan O’Donnell, Omer Tamuz, and Li-Yang Tan.
\newblock Convergence, unanimity and disagreement in majority dynamics on
  unimodular graphs and random graphs.
\newblock {\em Stochastic Processes and their Applications}, 126(9):2719--2733,
  2016.

\bibitem{bh}
Simon~R. Broadbent and John~M. Hammersley.
\newblock Percolation processes: I. crystals and mazes.
\newblock In {\em Mathematical Proceedings of the Cambridge Philosophical
  Society}, volume~53, pages 629--641. Cambridge University Press, 1957.

\bibitem{csn}
Federico Camia, Emilio de~Santis, and Charles~M. Newman.
\newblock Clusters and recurrence in the two-dimensional zero-temperature
  stochastic {Ising} model.
\newblock {\em The Annals of Applied Probability}, 12(2):565--580, 2002.

\bibitem{cns}
Federico Camia, Charles~M. Newman, and Vladas Sidoravicius.
\newblock Approach to fixation for zero-temperature stochastic {Ising} models
  on the hexagonal lattice.
\newblock In {\em In and Out of Equilibrium}, pages 163--183. Springer, 2002.

\bibitem{cm}
Pietro Caputo and Fabio Martinelli.
\newblock Phase ordering after a deep quench: the stochastic {Ising} and hard
  core gas models on a tree.
\newblock {\em Probability theory and related fields}, 136(1):37--80, 2006.

\bibitem{fss}
Luiz~R. Fontes, Roberto~H. Schonmann, and Vladas Sidoravicius.
\newblock Stretched exponential fixation in stochastic {Ising} models at zero
  temperature.
\newblock {\em Communications in mathematical physics}, 228(3):495--518, 2002.

\bibitem{gandolfi}
Alberto Gandolfi, Michael Keane, and Lucio Russo.
\newblock On the uniqueness of the infinite occupied cluster in dependent
  two-dimensional site percolation.
\newblock {\em The Annals of probability}, 16(3):1147--1157, 1988.

\bibitem{gh}
Yuval Ginosar and Ron Holzman.
\newblock The majority action on infinite graphs: strings and puppets.
\newblock {\em Discrete Mathematics}, 215(1-3):59--71, 2000.

\bibitem{grimmett}
Geoffrey Grimmett.
\newblock In {\em Percolation}. Springer, 1999.

\bibitem{harris}
Theodore~E. Harris.
\newblock A correlation inequality for {Markov} processes in partially ordered
  state spaces.
\newblock {\em The Annals of Probability}, pages 451--454, 1977.

\bibitem{howard}
C.~Douglas Howard.
\newblock Zero-temperature {Ising} spin dynamics on the homogeneous tree of
  degree three.
\newblock {\em Journal of applied probability}, 37(3):736--747, 2000.

\bibitem{moran}
Gadi Moran.
\newblock On the period-two-property of the majority operator in infinite
  graphs.
\newblock {\em Transactions of the American Mathematical Society},
  347(5):1649--1667, 1995.

\bibitem{morris}
Robert Morris.
\newblock Zero-temperature {Glauber} dynamics on $\mathbb{Z}^{d}$.
\newblock {\em Probability theory and related fields}, 149(3-4):417--434, 2011.

\bibitem{nns}
Seema Nanda, Charles~M. Newman, and Daniel~L. Stein.
\newblock Dynamics of {Ising} spin systems at zero temperature.
\newblock {\em Translations of the American Mathematical Society-Series 2},
  198:183--194, 2000.

\bibitem{tt}
Omer Tamuz and Ran~J. Tessler.
\newblock Majority dynamics and the retention of information.
\newblock {\em Israel Journal of Mathematics}, 206(1):483--507, 2015.

\bibitem{tassion}
Vincent Tassion.
\newblock Crossing probabilities for {Voronoi} percolation.
\newblock {\em The Annals of Probability}, 44(5):3385--3398, 2016.

\end{thebibliography}

\end{document}